\documentclass[12pt]{amsart}

\usepackage{amssymb}
\usepackage{bbm}
\usepackage{verbatim}
\usepackage{graphicx,color}
\usepackage{enumerate}
\usepackage[all]{xy}
\usepackage{tabu,makecell}
\usepackage{enumitem}
\usepackage{lipsum}

\usepackage[hyperindex=true,plainpages=false,colorlinks=false,pdfpagelabels]{hyperref}

\theoremstyle{plain}
\newtheorem{Theorem}{Theorem}
\newtheorem*{Theorem*}{Theorem}

\newtheorem{Lemma}{Lemma}
\newtheorem{Corollary}{Corollary}

\newtheorem*{Cor*}{Corollary}

\theoremstyle{definition}
\newtheorem{Definition}{Definition}
\newtheorem{Remark}{Remark}
\newtheorem{Example}{Example}
\newtheorem{Example*}{Example}
\newtheorem*{Remark*}{Remark}
\newtheorem*{Conjecture*}{Conjecture}

\DeclareMathOperator{\Conf}{Conf}
\DeclareMathOperator{\End}{End}
\DeclareMathOperator{\Hom}{Hom}

\DeclareMathOperator{\SO}{SO}

\DeclareMathOperator{\Id}{id}

\DeclareMathOperator{\vol}{vol}

\DeclareMathOperator{\Jac}{Jac}

\DeclareMathOperator{\Iso}{Iso}

\renewcommand{\Im}{\operatorname{Im}}
\renewcommand{\Re}{\operatorname{Re}}

\DeclareMathOperator{\dbar}{\bar\partial}

\DeclareSymbolFont{bbold}{U}{bbold}{m}{n}
\DeclareSymbolFontAlphabet{\mathbbold}{bbold}

\newcommand{\R}{\mathbb{R}}

\newcommand{\C}{\mathbb{C}}
\newcommand{\N}{\mathbb{N}}
\newcommand{\Z}{\mathbb{Z}}

\renewcommand{\H}{\mathbb{H}}

\newcommand{\ii}{\mathbbm i}
\newcommand{\jj}{\mathbbm j}

\newcommand{\D}{\mathcal D}
\newcommand{\spin}{\Sigma}

\graphicspath{{./figs/}}

\begin{document}

\title{Finding conformal and  isometric immersions of surfaces}

\author{Albert Chern}
\author{Felix Kn{\"o}ppel}
\author{Franz Pedit}
\author{Ulrich Pinkall}
\author{Peter Schr{\"o}der}



\date{\today}
\thanks{Authors supported by SFB Transregio 109 ``Discretization in Geometry and Dynamics'' at Technical University Berlin. Third author partially supported by an RTF grant from the University of Massachusetts Amherst. Fifth author partially supported by the Einstein Foundation. Software support for images provided by SideFX. We thank Stefan Sechelmann for the abstract hyperbolic triangulated surface used in Figure~1.}

\maketitle


\section{Introduction}
The notion of an abstract Riemannian manifold raises the question of whether every such manifold can be isometrically realized as a submanifold of Euclidean space. This problem has  been given an affirmative answer in the smooth category by Nash \cite{Nash:1954:CII}, provided that the codimension of the submanifold is sufficiently large. If one asks the more specific question of whether a given 2-dimensional Riemannian manifold $(M,g)$  can be isometrically immersed into Euclidean $3$-space, not too much is known.  There are  general local existence results  for real analytic metrics \cite{Spivak:1975:CIDG} and for smooth metrics under certain curvature assumptions \cite{Lin:1986:LIE}. Non-existence results are easier to come by: for instance,  Hilbert's classical result that the hyperbolic plane does not admit an  isometric immersion  into  $\R^3$, or the fact that a compact non-positively curved $2$-dimensional Riemannian manifold cannot be isometrically immersed into $\R^3$. 

\begin{figure}[h]
	\center
	\includegraphics[width=.85\textwidth]{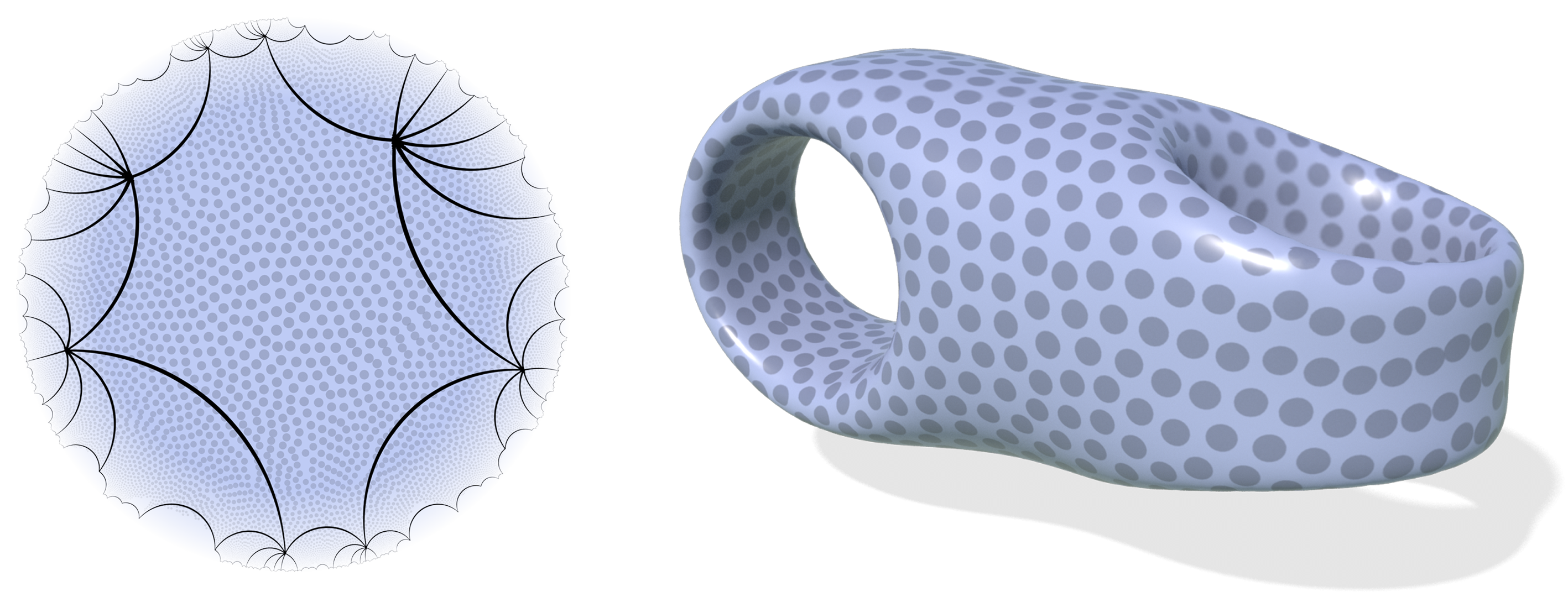}
	\caption{\label{fig:genus2-hyperbolic}
	A smooth, almost isometric immersion of the Riemannian surface of genus 2 with constant curvature shown on the left found by the algorithm in \cite{Chern:2018:shape}.
	}
\end{figure}

Surprisingly though, if one relaxes the smoothness of the immersion, every $2$-dimen\-sional Riemannian manifold $(M,g)$ admits a $C^1$-isometric immersion $f\colon M\to \R^3$  into Euclidean space \cite{Nash:1954:CII,Kuiper:1955:CII,Gromov:1986:PDR}.  Unfortunately, neither the original existence proofs nor the recent explicit constructions of such isometric immersions~\cite{Borrelli:2012:FTT,Borrelli:2013:IES}  reflect much of the underlying geometry of $(M,g)$, as shown in Figure~\ref{fig:C1-flattorus}.

On the other hand, there are piecewise linear embeddings of a flat torus which make visible its intrinsic geometry (see Figure~\ref{fig:linear-flatorus}). 
In a more general vein, one could attempt to find isometric immersions in the class of piecewise smooth immersions  $f\colon M\to \R^3$, that is, local topological embeddings whose restrictions to the closed faces of a triangulation of $M$ are smooth. Experiments carried out with a recently developed numerical algorithm~\cite{Chern:2018:shape} provide support of the following.
\begin{Conjecture*}
Given a Riemannian surface $(M,g)$, there exists a piecewise smooth isometric immersion $f\colon M\to \R^3$  in each regular homotopy class. 
\end{Conjecture*}
The added detail---to realize a given intrinsic geometry within a prescribed regular homotopy class---is advantageous in applications to computer graphics~\cite{Chern:2018:shape} and also for the theoretical approach to the isometric immersion problem.  It is the latter which will be discussed in this paper. Our objective is to rephrase the isometric immersion problem of an oriented Riemannian surface $(M,g)$ into Euclidean space $\R^3$ as a variational problem with parameters whose minima, if they were to exist, converge (for limiting parameter values) to isometric immersions $f\colon M\to \R^3$ in a given regular homotopy class.  As was pointed out already, for a generic metric $g$ there will be no smooth isometric immersion into $\R^3$, let alone one within a prescribed regular homotopy class. But experiments with the aforementioned algorithm ~\cite{Chern:2018:shape} give some credence to our conjecture that there should be  minima in the larger class of piecewise smooth immersions.  Adjusting the parameters in our functional, the Willmore energy $\int H^2$, the averaged squared mean curvature of the immersion, is one of its contributors and hence immersions close to a minimizer will avoid excessive creasing.  This has the effect that potential minimizers of our functional reflect the intrinsic geometry of $(M,g)$ well, in contrast to the $C^1$-isometric immersions by Nash and Kuiper ~\cite{Nash:1954:CII,Kuiper:1955:CII,Gromov:1986:PDR}.

In order to explain our approach in more detail, we first relax the original problem to that of  finding a {\em conformal} immersion $f\colon M\to\R^3$ of a compact {\em Riemann surface} $M$ in a given regular homotopy class. It is known that such a conformal immersion always exists in the smooth category ~\cite{Garsia:1962:ASES,Ruedy:1971:EORS}, and hence our variational problem will have a minimizer if we turn off the contribution from the Willmore energy. But keeping the Willmore energy in the functional has the effect that potential minimizers will minimize the  Willmore energy in a given conformal and regular homotopy class, that is, will be constrained Willmore minimizers. There are partial characterizations of constrained Willmore minimizers when $M$ has genus one ~\cite{KuwSch:2001:WFSIE, SchNdy:2014:CCWM, HelNdy:2017:ECWM}, but hardly anything---besides existence ~\cite{KuwSch:2001:WFSIE} if the Willmore energy is below $8\pi$---is known in higher genus, even though there are some conjectures \cite{HelPed:2017:TCWC}. One of our future goals is to develop a discrete algorithm based on the approach outlined in this paper to find conformal immersions of a  compact Riemann surface in a fixed regular homotopy class minimizing the Willmore energy.
 
Given a (not necessarily conformal)  immersion $f\colon M\to\R^3$, we can decompose its derivative $df\in\Omega^1(M,\R^3)$ uniquely into 
$df=\omega\circ B$
where  $\omega\in \Gamma(\Conf(TM,\R^3))$ is a conformal, nowhere vanishing $\R^3$-valued $1$-form and  $B\in\Gamma(\End(TM))$ is a positive, self-adjoint (with respect to any conformal metric) endomorphism with  $\det B=1$. Obviously $B=\Id$ if and only if $f$ is conformal. The space of conformal $1$-forms $\Conf(TM,\R^3)$ is a principal bundle with stretch rotations $\R_{+}{\bf SO}(3)$ as a structure group acting from the left. In particular, any two $\omega,\,\tilde{\omega}\in \Gamma(\Conf(TM,\R^3))$ are related via $\tilde{\omega}= h\, \omega$ for a unique $h\colon M\to  \R_{+}{\bf SO}(3)$. Two immersions $f,\,\tilde{f}\colon M\to\R^3$ are regularly homotopic if and only if their derivatives $df$ and $d\tilde{f}$ are homotopic~\cite{Smale:1959:CIT,Hirsch:1959:IOM}. The space of positive, self-adjoint bundle maps $B\in\Gamma(\End(TM))$ with $\det B=1$ is contractible, and we obtain  the equivalent reformulation that $f,\, \tilde{f}\colon M\to\R^3$ are regularly homotopic if and only if their corresponding $\omega$ and $\tilde{\omega}$ are homotopic in $\Gamma(\Hom(TM,\R^3))$. 

As will be detailed in Section~\ref{sec:spinbundles}, any nowhere vanishing $\omega\in  \Gamma(\Conf(TM,\R^3))$ induces a spin bundle $L\to M$ and $\omega=(\psi,\psi)$ for a unique (up to sign) nowhere vanishing section $\psi\in\Gamma(L)$ where $(\cdot,\cdot)\colon L\times L\to \Hom(TM,\H)$ denotes the spin pairing. Since homotopic $\omega\in  \Gamma(\Conf(TM,\R^3))$ give rise to isomorphic spin bundles, we obtain a description of regular homotopy classes of immersions via isomorphism classes of their induced spin bundles $L\to M$. If the genus of $M$ is $p$, there are $2^{2p}$ many non-isomorphic spin bundles and hence  $2^{2p}$ many regular homotopy classes of immersions $f\colon M\to \R^3$ (see Figure~\ref{fig:tori-spinstructures}). 

\begin{figure}[b]
	\center
	\includegraphics[width=.9\textwidth]{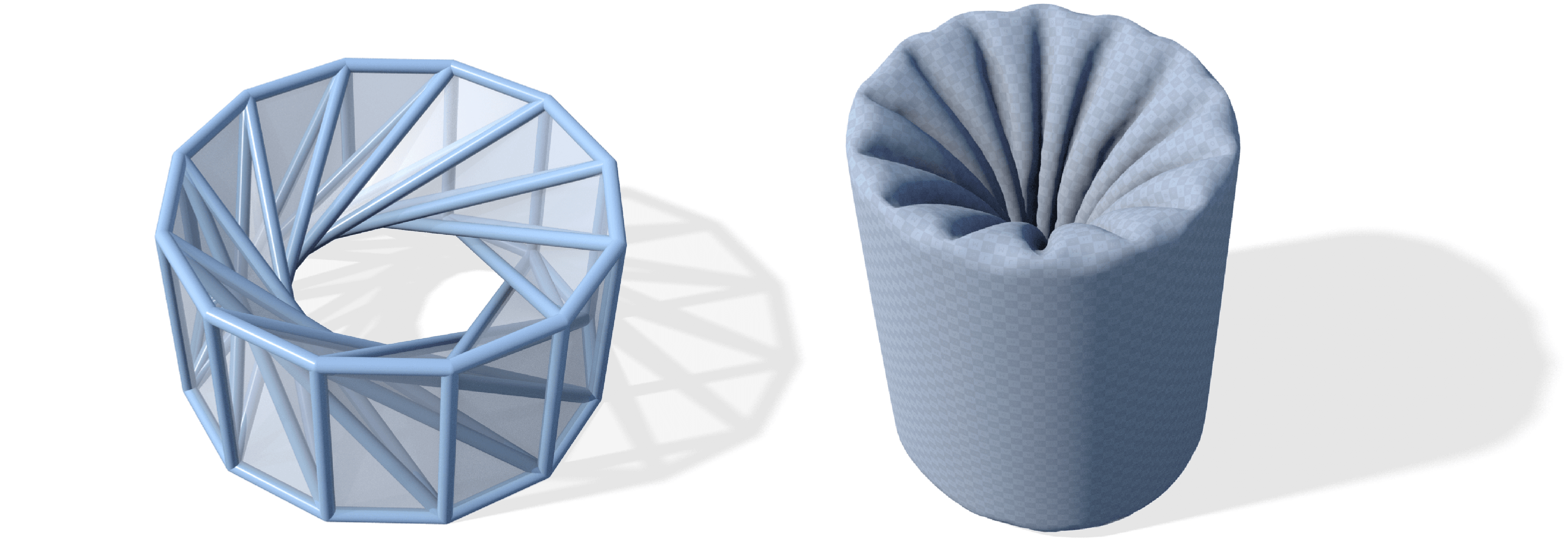}
	\caption{\label{fig:linear-flatorus} The angular defects of the embedded, piecewise linear torus shown on the left are the same at all vertices and their sum is zero, implying  that the induced metric is flat. The right image shows a smooth, almost isometric immersion of another flat torus found by the algorithm in \cite{Chern:2018:shape}.}
\end{figure}

Fixing  a regular homotopy class, that is, a spin bundle $L\to M$, our aim is to find a nowhere vanishing section $\psi\in\Gamma(L)$ in such a way that the $\R^3$-valued conformal $1$-form 
$\omega=(\psi,\psi) \in \Gamma(\Conf(TM,\R^3))$ is exact. In this case, the primitive $f\colon M\to \R^3$ of $\omega=df$ is a conformal immersion in the given regular homotopy class. We show (see also \cite{icmpaper, klassiker}) that the closedness of $\omega=(\psi,\psi)$ is equivalent  to the non-linear Dirac equation
\begin{equation}\label{eq:Dirac}
\dbar \psi +\tfrac{1}{2}H J\psi (\psi, \psi)=0,
\end{equation}
where $\dbar$ is the Dirac structure (see Lemma~\ref{lem:Diracstructure}) on the spin bundle $L$. The function $H\colon M\to \R$ is the mean curvature, calculated  with respect to the induced metric $|df|^2$, of the resulting conformal  immersion $f\colon \tilde{M}\to \R^3$ on the universal cover with translation periods.  

As we shall discuss in Section~\ref{sec:variational}, the Dirac equation~\eqref{eq:Dirac} can be given a variational characterization: for non-negative coupling constants $\epsilon=(\epsilon_1,\epsilon_2,\epsilon_3)$, we consider the family of variational problems $E_{\epsilon}\colon \Gamma(L^{\times})\to\R$ on nowhere vanishing sections of $L$ given by 
\begin{equation}\label{eq:functional}
E_{\epsilon}(\psi)=\epsilon_1\int_M \tfrac{\langle*\dbar\psi\wedge\dbar\psi\rangle}{|\psi|^2}+
(\epsilon_2-\epsilon_1)\int_M\tfrac{{\langle *\dbar\psi\wedge\psi(\psi,\psi)\rangle}^2}{|\psi|^4}+
(\epsilon_3-\epsilon_1)\int_M\tfrac{{\langle *\dbar\psi\wedge J\psi(\psi,\psi)\rangle}^2}{|\psi|^4}\,.
\end{equation}
Here  $|\cdot|^2\colon L\to |K|$ denotes the half-density valued quadratic form $|\psi|^2=|(\psi,\psi)|$ on $L$ and 
$\langle \cdot, \cdot \rangle\colon L\to |K|$ is the half-density valued inner product obtained via polarization.  The complex structure $*$ on $TM^{*}$ is  the {\em negative} of the Hodge-star on $1$-forms.
It is worth noting that the functional $E_{\epsilon}$ is conformally invariant, that is, well-defined on the Riemann surface $M$, and independent on constant scalings of $\psi$. 
In particular, we could normalize $\psi$  by restricting to the $L^4$-sphere of sections satisfying  $\int_M|\psi|^4=1$. 

The last integral  in \eqref{eq:functional} turns out to be the Willmore functional $\int_M H^2|\psi|^4$ and the first two integrals measure, in $L^2$, the failure of the non-linear Dirac equation~\eqref{eq:Dirac} to hold. Thus, for $\epsilon_3=0$ and $\epsilon_1, \epsilon_2>0$, the functional attains its minimum value $E_{\epsilon}(\psi)=0$ at nowhere vanishing sections $\psi$ which correspond to---in general rather singular---conformal immersions. 
It is therefore conceivable that minimizers of $E_{\epsilon}$ for $\epsilon_3>0$, which has the effect of keeping the Willmore energy as a regularizer, will converge as $\epsilon_3$ tends to zero to  smooth conformal immersions of $M$ minimizing the Willmore energy, that is, constrained Willmore surfaces.
Since the Dirac equation only guarantees the closedness of $(\psi,\psi)$, the resulting conformal immersion given by $df=(\psi,\psi)$ generally will have translation periods which are controlled by adding the squared lengths of the period integrals $|\int_{\gamma} (\psi,\psi)|^2$  to the functional~\eqref{eq:functional}.


As it turns out, this strategy works surprisingly well~\cite{Chern:2018:shape} when searching for isometric immersions $f\colon M\to\R^3$ of a compact, oriented Riemannian surface $(M,g)$. Since the conformal class of $g$ gives $M$ the structure of a Riemann surface, we can consider the family of functionals~\eqref{eq:functional} with the additional  {\em isometric constraint} 
\[
|\psi|^4=g\,.
\]
 Then the resulting minimizers under the above described  procedure will be isometric immersions $f\colon M\to\R^3$ whose Willmore energy is ``small''.  The resulting surfaces provide examples of how piecewise smooth, and sometimes even smooth, isometric immersions of compact Riemannian surfaces might look, as can be seen in  Figures~\ref{fig:genus2-hyperbolic}, \ref{fig:linear-flatorus}, and \ref{fig:genus2-abeljac}. 

We should point out that spinorial descriptions of surfaces have been applied to a variety of problems, both  in the discrete~\cite{Crane:2011:STD,Crane:2013:RFC,Zi:2018:DIEDO,Hoffmann:2018:Discrete} and smooth settings~\cite{kamberov, icmpaper, klassiker, KusSch:1996:SRSS,Taimanov:2008:S3DLG, Konopelchenko:2000:WR,heller-lawson}. The present paper is novel as it focuses on the spinorial construction of  conformal and isometric immersions of surfaces in $\R^3$ from a purely intrinsic point of view.

\section{Spin bundles and regular homotopy classes}\label{sec:spinbundles} 
Given a (not necessarily oriented or compact) $2$-dimensional manifold $M$, we will discuss how to relate a regular homotopy class of immersions $f\colon M \to \R^3$ and an isomorphism class of  spin bundles $L$ over $M$. The material is somewhat folklore \cite{Smale:1959:CIT,Hirsch:1959:IOM,Pinkall:1985:RHC,icmpaper,klassiker}, even though  there seems to be no single source one could reference. Recall that two smooth immersions $f,\tilde{f}\colon M\to \R^3$  are {\em regularly homotopic} if and only if there is a smooth homotopy via immersions $f_{t}\colon M\to \R^3$ with $f_0=f$ and $f_1=\tilde{f}$. It is well known~\cite{Smale:1959:CIT,Hirsch:1959:IOM} that two immersions $f$ and $\tilde{f}$ are regularly homotopic if and only if their derivatives $df$ and $d\tilde{f}$ are smoothly homotopic as sections in $\Hom(TM,\R^3)$. 
\begin{Definition}\label{def:spinbundle}
 A {\em spin bundle} over $M$  is a right quaternionic line bundle $L\to M$ together with a  non-degenerate quaternionic skew-Hermitian pairing
\begin{equation}\label{eq:pairing}
(\cdot,\cdot)\colon L\times L\to \Hom(TM,\H)\,, 
\end{equation}
which we refer to as a {\em spin pairing}. 

Two spin bundles $L,\tilde{L}\to M$ are isomorphic if there is a bundle isomorphism $T\colon L\to\tilde{L}$ intertwining their respective spin pairings.
\end{Definition}
Later in the paper we use the extension of the spin pairing to  the $2$-form valued pairing
\[
(\cdot,\cdot)\colon \Hom(TM,L) \times L\to \Lambda^2  TM^{*}\otimes\H\,,\qquad (\mu,\psi)_{X,Y}:=(\mu_X,\psi)_Y-(\mu_Y,\psi)_X
\]
obtained by inserting an $L$-valued $1$-form $\mu$ on the left, where $X,Y\in TM$. Requiring the skew-Hermitian property 
\[
\overline{(\mu,\psi)}=-(\psi,\mu)
\] 
to pertain in this scenario, necessitates the analogous definition 
\[
(\cdot,\cdot)\colon  L\times \Hom(TM,L)\to \Lambda^2  TM^{*}\otimes\H\,,\qquad (\psi,\mu)_{X,Y}=(\psi,\mu_X)_Y-(\psi,\mu_Y)_X
\]
when inserting the $L$-valued $1$-form $\mu$ on the right.

Note that by transversality a quaternionic line bundle $L\to M$ always has a nowhere vanishing smooth section $\psi\in\Gamma(L)$.  We denote the (right) $\H^{\times}$ principal bundle obtained by removing the zero-section of $L$ by $L^{\times}$, then $\Gamma(L^{\times})$ is the space of nowhere vanishing sections. The following are immediate consequences from the definition of a spin bundle:
\begin{enumerate}
\item
$\overline{(\psi,\psi)}=-(\psi,\psi)$, so that $(\psi,\psi)\in \Hom(TM,\R^3)$ is an $\R^3$-valued $1$-form where we identify $\R^3=\Im \H$.
\vspace{1em}
\item
Any two sections $\psi,\varphi\in\Gamma(L^{\times})$ scale by a nowhere vanishing function $\lambda\in C^{\infty}(M,\H)$ and hence
\[
 (\varphi, \varphi)=(\psi\lambda,\psi\lambda)=\bar{\lambda}(\psi,\psi)\lambda
 \]
 are pointwise related by a stretch rotation in $\R^3$.
 Therefore, the Riemannian metrics $|(\varphi, \varphi)|^2 =|\lambda|^2|(\psi,\psi)|^2$ are conformally equivalent and $M$ inherits a conformal structure, rendering $(\psi,\varphi)\in\Omega^1(M,\H)$ conformal.  

Whence we observe that an oriented $M$ becomes a Riemann surface in which case we denote its complex structure by $*\colon TM^*\to TM^*$, the negative of the Hodge-star operator on $1$-forms. Since $(\psi,\varphi)\in \Hom(TM,\H)$  now are conformal $1$-forms there is, at each point of $M$, a unique (quaternionic linear) complex structure $J\in\End(L)$ such that 
  \begin{equation}\label{eq:J-relate}
 *(\psi,\varphi)=(J\psi,\varphi)=(\psi, J\varphi)
 \end{equation}
 for all $\psi,\varphi\in L$. In particular, $L$ becomes a rank $2$ complex vector bundle.
 \end{enumerate}
Note that if we had started from a Riemann surface in our Definition~\ref{def:spinbundle} of a spin bundle, the existence of the complex structure $J\in\Gamma(\End(L))$ and this last compatibility relation~\eqref{eq:J-relate} would become part of the axioms.
\begin{Example}\label{ex:inducedspin}[The induced spin bundle]
Let $f\colon M\to\R^3$ be a (not necessarily conformal) immersion of a conformal surface. 
We can uniquely decompose the derivative
\begin{equation}\label{eq:polar}
df=\omega\circ B
\end{equation}
into a nowhere vanishing conformal $1$-form $\omega\in\Gamma(\Conf(TM,\R^3))$ and a positive, self-adjoint (with respect to any conformal 
metric) bundle isomorphism $B\in\Gamma(TM)$ with $\det B=1$. Note that $B=\Id$ if and only if $f$ is conformal. 
We define the induced spin bundle to be the trivial quaternionic line bundle 
\[
L_f:=M\times \H
\]
together with the  spin pairing 
\[
(\psi,\varphi)= \bar{\psi}\,\omega\,\varphi\,.
\]
Note that, by construction, $\omega=(1,1)$ with $1\in\Gamma(L_f)$  the constant section. 

 \begin{figure}
	\center
	\includegraphics[width=.85\textwidth]{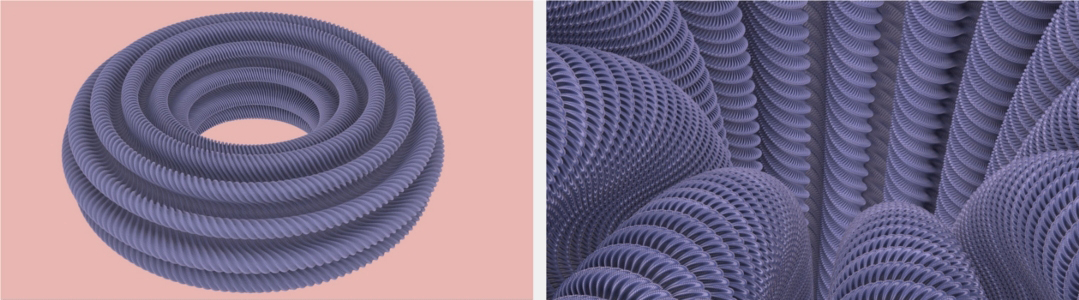}
	\caption{\label{fig:C1-flattorus}A square flat torus can be isometrically \(C^1\)-embedded into Euclidean 3-space \cite{Borrelli:2013:IES}. Pictures by the H{\'e}v{\'e}a project.}
\end{figure}

In case $M$ is oriented, and thus a Riemann surface, the immersion $f\colon M\to\R^3$ has an oriented normal $N\colon M\to S^2$, which, viewed as an imaginary quaternion, satisfies $N^2=-1$.  Then the conformal $1$-form $\omega\in\Gamma(\Conf(TM,\R^3))$  satisfies $*\omega=N\omega$ 
and $J\varphi:=-N\varphi$, for $\varphi\in L_f$,  is the unique complex structure on $L_f$ which satisfies the  compatibility properties \eqref{eq:J-relate}. 
\end{Example}
\begin{Theorem}\label{thm:spin-homotop}
The assignment $f\mapsto L_f$ is a bijection between regular homotopy classes of immersions $f\colon M\to\R^3$  and isomorphism classes of spin bundles $L\to M$. 
\end{Theorem}
\begin{proof}
Let $f_t\colon M\to \R^3$ be a regular homotopy between two (not necessarily conformal) immersions $f_0=f$ and $f_1=\tilde{f}$.  Then their derivatives $df, d\tilde{f}$ are homotopic by~\cite{Smale:1959:CIT,Hirsch:1959:IOM}, and therefore, by uniqueness of \eqref{eq:polar}, we have a homotopy $ \omega_t$ in $\Gamma(\Conf(TM,\R^3))$ connecting $\omega_0=\omega$ and $\omega_1=\tilde{\omega}$. Since $\Conf(TM,\R^3)$ is an $\R_{+}{\bf SO}(3)$ principal bundle, there exists a path $h\in C^{\infty}([0,1],\R_{+}{\bf SO}(3))$, starting at the identity $h_0=1$, with $\omega_t=h_t\,\omega$. Whence we conclude that there is a lift $\lambda\colon [0,1]\to \H^{\times}$ of $h$ with $\omega_t=\bar{\lambda}_t\omega \lambda_t$, and in particular we have
\[
\tilde{\omega}=\bar{\lambda}_1\omega \lambda_1\,.
\]
This last implies that the map $\varphi\mapsto T(\varphi):=\lambda_1\varphi$ is a isomorphism between the induced spin bundles $L_f$ and $L_{\tilde{f}}$.

In order to show the converse, let $L\to M$ be a spin bundle and choose a nowhere vanishing section $\psi\in\Gamma(L^{\times})$. Then 
$\omega=(\psi,\psi)\in\Gamma(\Conf(TM,\R^3))$ is a maximal rank $2$  conformal bundle map. By Smale's theorem \cite{Smale:1959:CIT}, there exists an immersion $f\colon M\to\R^3$ with $df$ homotopic to $\omega$ in $\Gamma(\Hom(TM,\R^3))$. From what was said before, we can conclude that $L\cong L_f$. Since all sections of $\Gamma(L^{\times})$ are homotopic, the regular homotopy class of the resulting immersion is independent of the nowhere vanishing section chosen. In particular, immersions constructed from isomorphic spin bundles are regularly homotopic. 
\end{proof}
So far, we were mainly concerned with the differential topological properties of spin bundles. Understanding how to construct conformal and isometric immersions  from spin bundles, we additionally need to  investigate their holomorphic aspects. Let $L\to M$ be a spin bundle over a Riemann surface $M$, in which case the spin pairing \eqref{eq:pairing} is compatible by \eqref{eq:J-relate} with the complex structures on $M$ and $L$. To fix notations and for future use, we list a number of properties of spin bundles over a Riemann surface that follow immediately from their definition. 
\begin{enumerate}
\item 
The complex line subbundles 
\[
E_{\pm}=\{\varphi\in L\,;\, J\varphi=\pm\varphi  i\}\subset L
\]
 are isomorphic via quaternionic multiplication by $j$ on the right, and thus as a complex rank $2$ bundle $L\cong E\oplus E$ is isomorphic to the double of the complex line bundle $E= E_{+}\cong E_{-}$. This isomorphism is also quaternionic linear provided $E\oplus E$ has the right quaternionic structure given by the Pauli matrices. 
\item
The spin pairing \eqref{eq:pairing} restricts to a non-degenerate complex pairing $E\times E\to K$ with values in the canonical bundle $K$ of $M$, exhibiting $E\to M$  as a complex spin bundle, that is $E^2\cong K$.  The holomorphic structure $\dbar^K$ of the canonical bundle is given by the exterior derivative $d$ on $\Gamma(K)=\Omega^{1,0}(M,\C)$.  The isomorphism $E^2\cong K$ induces a unique holomorphic structure $\dbar^{E}$ on $E$  such that $\dbar^K=\dbar^E\otimes\dbar^E$ or, equivalently, 
\[
d (\psi,\varphi)=(\dbar^{E}\psi,\varphi)+(\psi, \dbar^{E} \varphi)
\]
for $\psi,\varphi\in\Gamma(E)$.  In particular, if  $M$ is compact,  $\deg E= p-1$ is half of the degree of the canonical bundle, where $p=\text{genus}\,M$. Furthermore, $E$  with $\dbar^E$ is a holomorphic spin bundle and since there are $2^{2p}$ many holomorphic square roots of the canonical bundle $K$---the half lattice points in the Picard torus of isomorphism classes of degree $p-1$ holomorphic line bundles---there are $2^{2p}$ many isomorphism classes of holomorphic  spin bundles $E\to M$ over a compact  Riemann surface.
\end{enumerate}
\begin{Lemma}\label{lem:Diracstructure}
Let $L\to M$ be a spin bundle. Then there exists a unique operator 
\[
\dbar\colon \Gamma(L)\to\Gamma(\bar{K}L)
\]
called the {\em Dirac structure}, with the following properties. 
\begin{enumerate}
\item
$\dbar$ is complex linear, that is $[J,\dbar]=0$.
\item
$\dbar$ satisfies the product rule 
\[
\dbar(\psi\lambda)=(\dbar \psi)\lambda+ (\psi \,d\lambda)^{0,1}
\]
over quaternion valued functions $\lambda\in C^{\infty}(M,\H)$, where $(\cdot)^{0,1}$ denotes the usual type decomposition of complex vector bundle valued $1$-forms.  In particular, $\dbar$ is a (right) quaternionic and (left) complex linear first order elliptic operator.
\item
$\dbar$ is compatible with the spin pairing
\begin{equation}\label{eq:d-comp}
d(\psi,\varphi)=(\dbar\psi,\varphi)+(\psi, \dbar\varphi)\,.
\end{equation}
\end{enumerate}
\end{Lemma}
\begin{proof}
Since $L\cong E\oplus E$ is the double of a complex holomorphic spin bundle $E$, the operator $\dbar:=\dbar^{E}\oplus\dbar^{E}$ can be shown to fulfill the requirements of the lemma. Any other operator satisfying the properties of the lemma has to be of the form $\dbar+\alpha$ with $\alpha\in\Gamma(\bar{K})$ a $1$-form of type $(0,1)$. But then \eqref{eq:d-comp} implies that $\alpha=0$.
\end{proof}
\begin{Corollary}
Let $M$ be a compact oriented surface of genus $p$. Then there are $2^{2p}$ many isomorphism classes of spin bundles $L\to M$ and therefore, by  Theorem~\ref{thm:spin-homotop}, also  $2^{2p}$ many regular homotopy classes of immersions $f\colon M\to\R^3$.
\end{Corollary}
\begin{proof}
We know that a spin bundle $L\to M$ induces a unique complex structure on $M$ and $L$.  Since $L\cong E\oplus E$ for a holomorphic spin bundle $E\to M$, we conclude that there are $2^{2p}$  many isomorphism classes of spin bundles $L\to M$.
\end{proof}

At this point it is helpful to briefly review the notion of a quaternionic holomorphic structure \cite{icmpaper} on a quaternionic line bundle  $L\to M$  over  a Riemann surface. Such a structure is given by an operator
\[
D\colon \Gamma(L)\to\Gamma(\bar{K}L)
\]
satisfying the product rule 
\[
D(\psi\lambda)=(D \psi)\lambda+ (\psi \,d\lambda)^{0,1}
\]
over quaternion valued functions $\lambda\in C^{\infty}(M,\H)$. Note that choosing $\lambda\in\H$ constant, the product rule implies that $D$ is quaternionic linear. 

If $L\to M$ is a spin bundle, then we can demand the quaternionic holomorphic structure to be compatible with the spin pairing.
\begin{Definition}\label{def:holospin}
Let $L\to M$ be a spin bundle over a Riemann surface. A quaternionic holomorphic structure $D\colon \Gamma(L)\to\Gamma(\bar{K}L)$ is called a {\em quaternionic holomorphic spin structure}, if $D$ is compatible with the spin pairing 
\begin{equation}\label{eq:D-compatible}
d(\psi,\varphi)=(D\psi,\varphi)+(\psi, D\varphi)
\end{equation}
where  $\psi,\varphi\in\Gamma(L)$.
\end{Definition}
Note that by Lemma~\ref{lem:Diracstructure}, the Dirac structure $\dbar$ on a spin bundle $L\to M$ is a quaternionic holomorphic spin structure, in fact the unique one commuting with the complex structure $J$ on $L$. 

The general quaternionic holomorphic spin structure $D$ will not commute with $J$, and therefore will have a decomposition 
\begin{equation}\label{eq:D-decompose}
D=D_{+}+D_{-}
\end{equation}
into $J$ commuting and $J$ anti-commuting parts. The component $D_{+}=\dbar+\alpha$, a complex holomorphic structure on $L$, differs from the Dirac structure $\dbar$ by  a $(0,1)$-form $\alpha\in\Gamma(\bar{K})$, where we identify $\End_{+}(L)\cong \underline{\C}$. The component $D_{-}\in\Gamma(\bar{K}\End_{-}(L))$  is a $(0,1)$-form with values in the complex antilinear endomorphisms $\End_{-}(L)$. 

In order to characterize quaternionic holomorphic spin structures, it is helpful to  identify  $\bar{K}\End_{-}(L)$ with {\em half-densities}. Recall that the bundle of half-densities is the real, oriented  line bundle $|K|\to M$, whose fiber over $x\in M$ is given by $|K|_x=\mathbb{R}\sqrt{g_x}$, where $g$ is a Riemannian metric in the conformal class of $M$. The half-density valued quadratic form 
\begin{equation}\label{eq:quadraticform}
|\cdot|^2:L\to |K|\,,\qquad |\psi|^2:=|(\psi,\psi)|
\end{equation}
on the spin bundle $L$ can be polarized to the non-degenerate, symmetric inner product 
\begin{equation}\label{eq:anglebrac}
\langle \cdot, \cdot \rangle\colon L\times L\to |K|\,.
\end{equation}
We frequently will identify $|K|^2\cong \Lambda^2 TM^*$ by assigning a metric $g$ its volume $2$-form $\vol_g$. Since $|\psi|^4\in\Gamma(|K|^2)$ for $\psi\in\Gamma(L)$, a spin bundle  carries the conformally invariant $L^4$-metric $\int_M |\psi|^4$ on $\Gamma(L)$. 
\begin{Lemma}\label{lem:eta}
Let $L\to M$ be a spin bundle over a Riemann surface. Then the complex line bundle $\bar{K}\End_{-}(L)$ is isomorphic to the complexified half-density bundle 
\[
|K|\otimes\C \cong \bar{K}\End_{-}(L)\colon U+VJ\mapsto (U+VJ)\eta
\]
where $\eta\in\Gamma(\bar{K}\End_{-}(L)|K|^{-1})$ is the nowhere vanishing section 
\[
\eta:=J\psi\frac{(\psi,\cdot)}{|\psi|^2}\,.
\]
Notice that $\eta$ is  well-defined independent of the choice of the nowhere vanishing section $\psi\in\Gamma(L^{\times})$.
\end{Lemma}
\begin{proof}
If $\tilde{\psi}=\psi\lambda$ is another nowhere vanishing section of $L$, then 
\[
\tilde{\psi}\frac{(\tilde{\psi},\cdot)}{|\tilde{\psi}|^2}=\psi\lambda\frac{\bar{\lambda}(\psi,\cdot)}{|\lambda|^2|{\psi}|^2}=\psi\frac{(\psi,\cdot)}{|\psi|^2}
\]
which shows that $\eta$ is well-defined. It remains to verify that $*\eta = -J\eta$, that is,  $\eta\in\Gamma(\bar{K}\End_{-}(L)|K|^{-1})$. 
Let $\psi\in\Gamma(L^{\times})$ be a nowhere vanishing section so that  $J\psi=\psi N$ with $N^2=-1$. Then, using the compatibility relation~\eqref{eq:J-relate}, we obtain
\[
*\eta=J\psi\frac{*(\psi,\cdot)}{|\psi|^2}=J\psi\frac{(J\psi,\cdot)}{|\psi|^2}=J\psi\frac{(\psi N,\cdot)}{|\psi|^2}=J\psi(-N)\frac{(\psi,\cdot)}{|\psi|^2}=-J\eta, 
\]
which finishes the proof of the lemma.
\end{proof}
With these preparations, we can now give a characterization of  quaternionic holomorphic spin structures, which also can be found in~\cite{icmpaper}, albeit from a slightly different perspective.
\begin{Lemma}\label{lem:spin-decompose}
Every quaternionic holomorphic spin structure $D$ on a spin bundle $L\to M$ over a Riemann surface  is of the form 
\[
D=\dbar+U\eta
\]
with $\dbar$ the Dirac structure and  $U\in\Gamma(|K|)$ a real half-density, the {\em Dirac potential}. 
\end{Lemma}

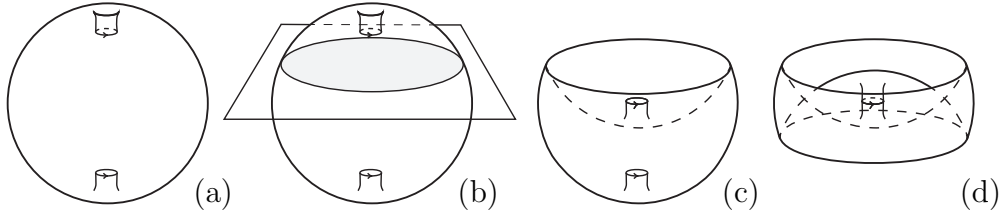
\begin{figure}
	\center
	\begin{picture}(385,80)(0,0)
		\put(-10,-10){\includegraphics[height=3.3cm]{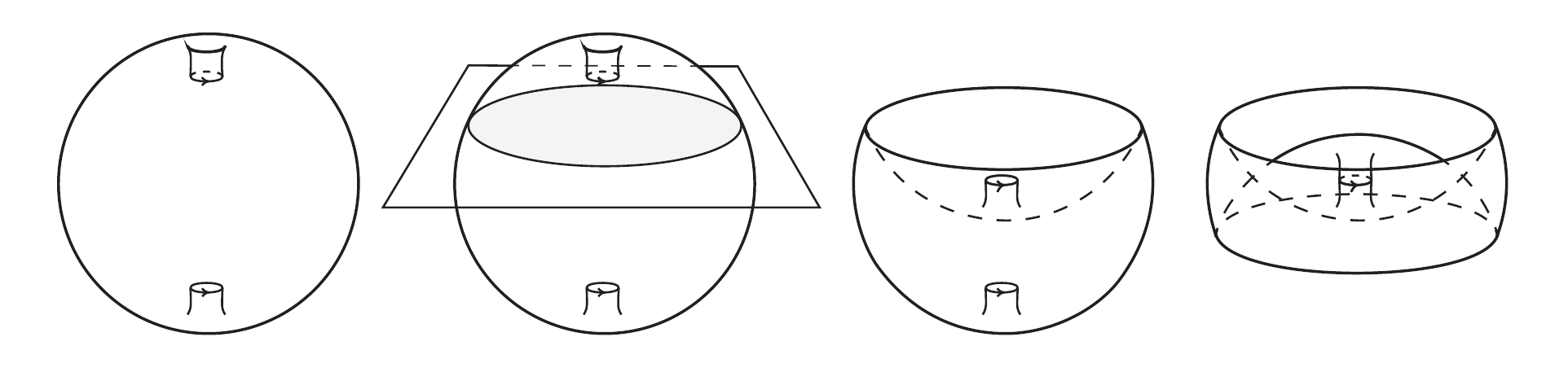}}
    		\put(75,0){(a)}
    		\put(175,0){(b)}
    		\put(275,0){(c)}
    		\put(365,0){(d)}
	\end{picture}
	\caption{\label{fig:tunneled sphere} Consider the Riemannian torus obtained by identifying the two boundary loops of the surface shown on the left (a). We believe that this torus does not admit a \(C^\infty\)-immersion into \(\mathbb{R}^3\). However, we can cut the surface by a plane (b) and reflect its upper part to obtain the surface (c). Applying this construction on the lower part of the surface results in the piecewise smooth isometric immersion of the torus (d).}
\end{figure}

\begin{proof}
From~\eqref{eq:D-decompose} and Lemma~\ref{lem:eta}, we know that 
\[
D=\dbar+\alpha+(U+JV)\eta
\]
with $\dbar$ the Dirac structure, $\alpha\in\Gamma(\bar{K})$, and $U,V\in\Gamma(|K|)$ half-densities. By Lemma~\ref{lem:Diracstructure} the Dirac structure already fulfills $d(\psi,\varphi)=(\dbar\psi,\varphi)+(\psi,\dbar\varphi)$. Thus, $D$ is a  quaternionic holomorphic spin structure if and only if the relation
\begin{equation}\label{eq:condition}
((\alpha+(U+JV)\eta)\psi,\varphi)+(\psi, (\alpha+(U+JV)\eta)\varphi)=0
\end{equation}
holds. To evaluate this last, we will use the following easy to verify identities. 
\begin{enumerate}
\item
Every $\alpha\in\Gamma(\bar{K})$ is of the form $\alpha=\beta+*\beta J$ for a unique real $1$-from $\beta\in\Omega^{1}(M,\R)$, and then
\[
(\alpha \psi,\varphi)+(\psi,\alpha\varphi)=2(\beta-*\beta)\wedge (\psi,\varphi)\,.
\]
\item
If $\psi\in\Gamma(L^{\times})$ is nowhere vanishing, then 
\[
|\psi|^2\left((\eta\psi,\varphi)+(\psi,\eta\varphi)\right)=0
\]
and
\[
|\psi|^2 \left( (J\eta\psi,\varphi)+(\psi,J\eta\varphi) \right)=2(\psi,\psi)\wedge(\psi,\varphi)\,.
\]
\end{enumerate}
Since the spin pairing is quaternionic Hermitian and $\eta$ is quaternionic linear, we may put
$\varphi=\psi\in\Gamma(L^{\times})$ in \eqref{eq:condition}.  Together with the above identities, \eqref{eq:condition} unravels to 
\begin{align*}
 0&=\left((\alpha+(U+JV)\eta)\psi,\psi\right)+(\psi, (\alpha+(U+JV)\eta)\psi)\\
 &=2(\beta-*\beta)\wedge (\psi,\psi)+ 2\frac{V}{|\psi|^2}(\psi,\psi)\wedge(\psi,\psi)\,.
 \end{align*}
Letting $J\psi=\psi N$ with $N^2=-1$, we deduce from the properties of the spin pairing~\eqref{eq:J-relate} that $(\psi,\psi)$ anti-commutes with $N$. Hence, the $\R^3$-valued $2$-forms $(\beta-*\beta)\wedge (\psi,\psi)$ and  $\frac{V}{|\psi|^2}(\psi,\psi)\wedge(\psi,\psi)$ in the last relation take values in  complementary subspaces of $\R^3$. Therefore, $D=\dbar+\alpha+(U+JV)\eta$ is a quaternionic holomorphic spin structure  if and only if $V=0$ and  $\beta-*\beta=0$, which implies $\beta=0$ and thus $\alpha=0$. 
\end{proof}
In Example~\ref{ex:inducedspin} we showed how  an immersion $f\colon M\to \R^3$ of a surface $M$ induces a spin bundle $L_f\to M$.
In case $f$ is a conformal immersion of a Riemann surface, the induced spin bundle $L_f$ additionally carries an induced quaternionic holomorphic spin structure.
\begin{Example}\label{ex:inducedD}[Induced quaternionic holomorphic structure]
Let $M$ be a Riemann surface and $f\colon M\to\R^3$ a conformal immersion. Since $B=\Id$ in the decomposition~\eqref{eq:polar}, the spin pairing of the induced spin bundle $L_f=M\times \H$ is given by
\[
(\psi,\varphi)=\bar{\psi}\,df\, \varphi
\]
for $\psi,\varphi\in\Gamma(L_f)$. In particular, $df=(1,1)$ for the constant section $1\in\Gamma(L_f^{\times})$. If $N\colon M\to S^2$ with $N^2=-1$ denotes the Gauss normal map of $f$, the conformality condition reads
\[
*df=N\,df = -df\,N \,.
\]
The  complex structure $J\in\Gamma(\End(L_f))$ on $L_f$ is given by the quaternionic linear endomorphism
\[
J\varphi:=-N\varphi
\] 
for $\varphi\in\Gamma(L_f)$ and the compatibility relations \eqref{eq:J-relate} hold.  There is a natural quaternionic holomorphic structure on $L_f$ given  by the  $(0,1)$-part of the trivial connection 
\[
D=d^{\,0,1}\colon \Gamma(L_f)\to\Gamma(\bar{K}L_f)\colon \varphi\mapsto \frac{1}{2}(d\varphi +J*d\varphi)\,.
\]
In order to verify that $D$ is in fact a quaternionic holomorphic spin structure, we need to assert the compatibility \eqref{eq:D-compatible} with the spin pairing
\begin{align*}
d(\psi,\varphi)&=d(\bar{\psi}\,df\,\varphi)=d\bar{\psi}\wedge df\, \varphi-\bar{\psi}\,df\wedge d\varphi
= \overline{d^{\,0,1}\psi}\wedge df\, \varphi-\bar{\psi}\,df\wedge d^{\,0,1}\varphi\\
&= (D\psi,\varphi)+(\psi,D\varphi)\,.
\end{align*}
Here we have used $\overline{d^{\,1,0}\psi}\wedge df= df\wedge d^{\,1,0}\varphi=0$ by type considerations.  Therefore, $D=d^{\,0,1}$
is a quaternionic holomorphic spin structure on $L_f$, and as such decomposes by Lemma~\ref{lem:spin-decompose} into 
\[
D=\dbar+U\eta
\]
with $\dbar$ the Dirac structure and $U\in\Gamma(|K|)$ the Dirac potential. One can easily compute \cite{icmpaper, klassiker, coimbra} that the Dirac potential is given by  the mean curvature half-density $U=\tfrac{1}{2}H|df|$, where $H\colon M\to \R$ is the mean curvature of $f$ calculated with respect to its induced metric $|df|^2$. 

The constant section $1\in\Gamma(L_f^{\times})$ lies in the kernel of $D=d^{\,0,1}$, which is expressed by the non-linear Dirac equation 
\[
\dbar 1+ \tfrac{1}{2}H J 1(1,1) =0\,.
\]
Here we used $|df|=|(1,1)|=|1|^2$ and the definition of $\eta$ in Lemma~\ref{lem:eta} with $\psi=1\in\Gamma(L_f^{\times})$.  The non-linear Dirac equation will be the starting point for our construction of conformal and isometric immersions in a given regular homotopy class.
\end{Example}

\section{Conformal and isometric immersions} \label{sec:variational}
Given a Riemann surface $M$, we want to construct a conformal immersion $f\colon M\to\R^3$ with small Willmore energy in a given regular  homotopy class. By Theorem~\ref{thm:spin-homotop}, a regular homotopy class is given by a choice of spin bundle $L\to M$ that  comes equipped with the Dirac structure $\dbar$ from Lemma~\ref{lem:Diracstructure}. Any nowhere vanishing section $\psi\in\Gamma(L^{\times})$ gives rise to a putative derivative  $(\psi,\psi)\in\Gamma(\Conf(TM,\R^3))$ of a conformal immersion in the regular homotopy class defined by $L$. The problem is that, in general, $(\psi,\psi)$ will not be closed, which is necessary for the existence of a conformal immersion $f\colon M\to\R^3$ satisfying $df=(\psi,\psi)$.  
\begin{Lemma}\label{lem:nonlinearDiracimm}
Let $L\to M$ be a spin bundle over the Riemann surface $M$ and $\psi\in\Gamma(L^{\times})$ a nowhere vanishing section of $L$.
Then the conformal 1-form $(\psi,\psi)\in\Gamma(\Conf(TM,\R^3))$ is closed if and only if $\psi$ solves the non-linear Dirac equation
\begin{equation}\label{eq:nonlinearDirac}
\dbar \psi + \frac{1}{2}HJ \psi (\psi,\psi) =0
\end{equation}
for some real valued function $H\colon M\to\R$.  

\begin{figure}[b]
	\center
	\includegraphics[width=.7\textwidth]{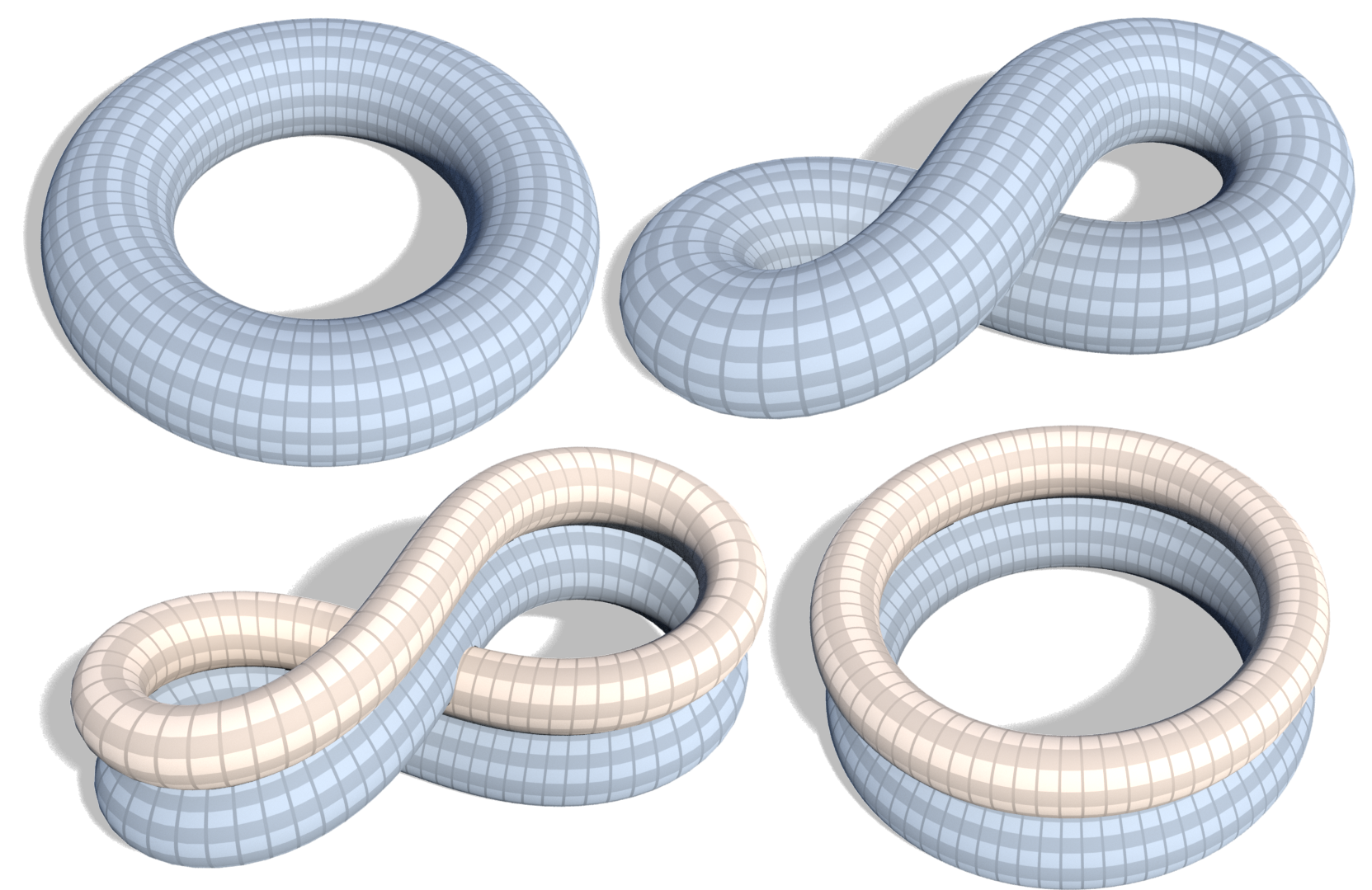}
	\caption{\label{fig:tori-spinstructures}The four regular homotopy classes of a torus.}
\end{figure}

The resulting conformal immersion $f\colon \widetilde{M}\to\R^3$ on the universal cover with translation periods satisfying $df=(\psi,\psi)$  has 
Gauss normal map $N\colon M\to S^2$ given by $J\psi=:-\psi N$. The induced spin bundle $L_f\cong L$ and the induced quaternionic holomorphic spin structure $d^{\,0,1}$ on $L_f$ corresponds, under this isomorphism,  to the quaternionic holomorphic spin structure $D= \dbar+\frac{1}{2}HJ \psi (\psi,\cdot)$. In particular, $H$ is the mean curvature of $f$ calculated with respect to the induced conformal metric $|df|^2=|\psi|^4$.
\end{Lemma}
\begin{Remark}
Strictly speaking, Examples~\ref{ex:inducedspin} and ~\ref{ex:inducedD} are stated for a (conformal) immersion $f\colon M\to\R^3$ without periods, but all constructions only use information about the derivative $df$. Whence a (conformal) immersion $f\colon \tilde{M}\to\R^3$ on the universal cover with translation periods induces a spin bundle $L_f\to M$  together with the induced quaternionic holomorphic structure $d^{\,0,1}$ over $M$.
\end{Remark}
\begin{proof}
From Lemma~\ref{lem:spin-decompose} we know that 
\[
D=\dbar+\frac{1}{2}HJ \psi (\psi,\cdot) =\dbar+U\eta \,,
\]
where $U=\tfrac{1}{2}H|\psi|^2$, is a quaternionic holomorphic spin structure. In particular,  $D$ is compatible \eqref{eq:D-compatible} with the spin pairing 
\[
d(\psi,\psi)=(D\psi,\psi)+(\psi,D\psi)\,.
\]
Therefore, if $\psi$ solves the non-linear Dirac equation, which, expressed in terms of $D$,  reads $D\psi=0$, the conformal $1$-form $(\psi,\psi)$ is closed. The converse follows from computations similar to the proof of Lemma~\ref{lem:spin-decompose}.  
Trivializing $L\cong M\times\H$ via the nowhere vanishing section $\psi\in\Gamma(L^{\times})$ provides the isomorphism $L\cong L_f$. 
The remaining statements follow from Example~\ref{ex:inducedD}.
\end{proof}
Given a spin bundle $L\to M$, our goal is to set up a variational problem with parameters
\[
E_{\epsilon}\colon \Gamma(L^{\times})\to\R
\]
on the space of nowhere vanishing sections of $L$, whose minima will give rise to  conformal immersions $f\colon M\to\R^3$.  From
the previous Lemma~\ref{lem:nonlinearDiracimm} we know that a nowhere vanishing section $\psi\in\Gamma(L^{\times})$ gives rise to a conformal immersion (with translation periods) whose derivative satisfies $df=(\psi,\psi)$, provided that $\psi$ solves the non-linear Dirac equation~\eqref{eq:nonlinearDirac}. In other words, $\psi$ has to satisfy
\[
\dbar \psi = U\eta\psi
\]
for some real half-density $U\in\Gamma(|K|)$. In general, since $\psi$ is nowhere vanishing,
$\dbar\psi= Q\psi$ for a $(0,1)$-form $Q\in\Gamma(\bar{K}\End(L))$ with values in the endomorphims of $L$. Decomposing $Q$ into the sum  $Q=Q_{+}+Q_{-}$ of the $J$ commuting part $Q_{+}=\alpha\in\Gamma(\bar{K})$ and the $J$ anti-commuting part $Q_{-}=(U+VJ)\eta$ with $U,V\in\Gamma(|K|)$ real half-densities, we obtain 
\begin{equation}\label{eq:decomposition}
\dbar\psi= \alpha\psi+(U+VJ)\eta\,\psi\,.
\end{equation}
Thus, $\psi$ solves the non-linear Dirac equation, if and only if $\alpha=0$ and $V=0$.  Before continuing, it is worthwhile to discuss the geometric implications of these conditions.
\begin{Remark}
The nowhere vanishing section $\psi\in\Gamma(L^{\times})$  gives rise to the conformal $1$-form $(\psi,\psi)\in\Gamma(\Conf(TM,\R^3))$ which is the putative derivative $df$ of a conformal immersion $f\colon M\to\R^3$.  From Lemma~\ref{lem:nonlinearDiracimm} the candidate for the Gauss normal map $N\colon M\to S^2$ of $f$  is given by $J\psi=:-\psi N$. We can decompose the rank 2 complex bundle $L\to M$ into the sum of complex line bundles, the $\mp N$ eigenspaces 
\[
L_{\pm}=\{\varphi\in L\,;\, J\varphi=\mp\varphi N\}
\]
of the complex structure $J\in\End(L)$. Then $L_{+}\subset L$ is a trivial line bundle via the nowhere vanishing section $\psi\in\Gamma(L_{+})$. Since
\[
*(\psi,\psi)=N(\psi,\psi)=- (\psi,\psi)N
\]
due to \eqref{eq:J-relate}, we have the well-defined complex line bundle isomorphism
\begin{equation}\label{eq:isomorph}
L_{+}^2\to K N^* TS^2\colon \psi^2\mapsto (\psi,\psi)\,.
\end{equation}
The Dirac structure induces complex holomorphic structures $\dbar_{\pm}$ on the summands $L=L_{+}\oplus L_{-}$. Since $*\eta=-J\eta$, the decomposition \eqref{eq:decomposition}
\[
\dbar\psi= \alpha\psi+(U+VJ)\eta\,\psi
\]
is adapted to the splitting $\bar{K}L= \bar{K}L_{+}\oplus \bar{K}L_{-}$. Therefore,  $\alpha=0$ if and only if the isomorphism \eqref{eq:isomorph}  is holomorphic, that is, $\dbar_{+} $ is the trivial holomorphic structure. In other words, $\alpha$ measures the failure of \eqref{eq:isomorph} to be holomorphic.  

Since $U\in\Gamma(|K|)$ is the putative mean curvature half-density, it remains to uncover the geometric meaning of the half density $V\in\Gamma(|K|)$ in \eqref{eq:decomposition}.  The derivative $dN\in\Omega^1(M,N^*TS^2)$  of the candidate Gauss map $N\colon M\to S^2$ can be decomposed into conformal and anti-conformal $\R^3$-valued $1$-forms
\[
dN=dN_{+}+dN_{-}= \frac{2}{|\psi|^2}(U(\psi,\psi) +V*(\psi,\psi)) + q\,.
\]
If $(\psi,\psi)=df$ were closed, than the latter would be the decomposition of the shape operator $dN$ into the trace part $H\,df$ and the trace-free part $q$, the Hopf differential. Therefore, $V=0$ is exactly the condition that the shape operator $dN$ is self-adjoint for one (and hence any) conformal metric on $M$. 

Incidentally, the above discussion of the geometric content of the decomposition \eqref{eq:decomposition} also gives an algorithmic answer to the question ``when is a map $N\colon M\to S^2$ from a {\em compact} Riemann surface $M$ the Gauss normal map of a conformal immersion?''  We first choose a spin bundle $L\to M$ which comes with a complex structure $J\in\Gamma(\End(L))$ compatible \eqref{eq:J-relate} with the Riemann surface structure of $M$. According to Theorem~\ref{thm:spin-homotop}, the spin bundle $L$ encodes one of the $2^{2p}$ regular homotopy classes of the resulting conformal immersion, where $p\in\N$ denotes the genus of $M$. The eigenspace decomposition
\[
L_{\pm}=\{\varphi\in L\,;\, J\varphi=\mp\varphi N\}
\]
defines the two complex line subbundles $L_{\pm}\subset L$ and we need $L_{+}\to M$ to admit a global nowhere vanishing section $\psi\in\Gamma(L_{+}^{\times})$. In other words, $L_{+}$ has to be trivializable which is equivalent to $\deg L_{+}=0$. Due to \eqref{eq:isomorph}
this last is guaranteed if and only if $\deg N= 1-g$, that is, $N$ has the correct degree required by the Gauss-Bonnet Theorem. Moreover, we have seen that $L_{+}$ needs to be holomorphically trivial, which puts $2p=\dim_{\R} \Jac(M)$ real conditions on $N$. Having chosen an $N$ satisfying those conditions, it remains to check whether the half-density $V\in\Gamma(|K|)$ in the decomposition \eqref{eq:decomposition} vanishes.
Note that globally the only remaining freedom is to rescale $\psi$ by a non-vanishing complex number $\lambda\in\C^{\times}$, which has the effect 
of a real scaling and a rotation of the complex half density $U+VJ$.  Provided that such a constant rotation renders this complex half density real, there will be a conformal immersion (with translation periods) $f\colon \tilde{M}\to \R^3$ whose Gauss normal map is given by $N$. 
\end{Remark}
After this brief interlude describing the geometric ramifications of the requirements $\alpha=0$ and $V=0$ in the decomposition \eqref{eq:decomposition}, which guarantee that the conformal $\R^3$-valued $1$-form $(\psi,\psi)\in\Gamma(\Conf(TM,\R^3))$  is closed, we shift towards the variational aspects of those conditions. On a compact Riemann surface $M$  the requirements $\alpha=0$ and $V=0$  are equivalent to the vanishing of the sum of their $L^2$-norms
$
\int_M *\bar{\alpha}\wedge \alpha+ \int_M V^2=0\,.
$
Put differently, our variational problem should be designed to measure, in $L^2$, the failure of $\psi\in\Gamma(L^{\times})$ to solve the non-linear Dirac equation~\eqref{eq:nonlinearDirac}. In the  following lemma we calculate the possible contributions to our functional. 
\begin{Lemma}\label{lem:contributions}
Let $L\to M$ be a spin bundle, $\dbar$ the Dirac structure on $L$, and $\psi\in \Gamma(L^{\times})$ a nowhere vanishing section. Then
we have the following expressions for the components of $\dbar\psi$ in the decomposition \eqref{eq:decomposition}:
\begin{enumerate}
\item
$\langle*\dbar\psi\wedge \dbar\psi\rangle=|\psi|^2(*\bar{\alpha}\wedge\alpha+|U|^2+|V|^2)$,
\item
$\langle*\dbar\psi\wedge \eta\psi\rangle=|\psi|^2 U$, 
\item
$\langle*\dbar\psi\wedge J\eta\psi\rangle=|\psi|^2 V$.
\end{enumerate}
\end{Lemma} 
\begin{proof} The real half-density valued inner product \eqref{eq:anglebrac} on the quaternionic line bundle $L$ can always be seen as the real part of a quaternionic Hermitian symmetric inner product. Thus, we have 
\[
\quad \langle \psi\lambda, \psi\mu\rangle=\Re(\bar{\lambda}\mu)|\psi|^2
\]
for $\lambda,\mu\in\H$ and $\psi,\varphi\in L$. Moreover, the compatibility~\eqref{eq:J-relate}  of the complex structure $J\in\Gamma(\End(L))$ with the spin pairing implies
\[
\langle \psi,\psi\rangle=\langle J\psi,J\psi\rangle\quad\text{and thus}\quad \langle \psi,J\psi\rangle=0\,.
\]
We therefore also have 
\[
\langle a\psi, b\psi\rangle=\Re(\bar{a}b)|\psi|^2
\]
for $a,b\in\C$. Recall that $\eta\psi=J\psi\frac{\omega}{|\psi|^2}$ with $\omega=(\psi,\psi)\in\Omega^1(M,\R^3)$ so that $\bar{\omega}=-\omega$.
The $(0,1)$-form  $\alpha\in\Gamma(\bar{K})$ can be written as $\alpha=\beta+*\beta J$ for a real $1$-form $\beta\in\Omega^1(M,\R)$.
Applying the above identities, after some calculations we obtain the following results. 
\begin{gather*}
\langle *\alpha\psi\wedge \alpha\psi\rangle=\Re(*\bar{\alpha}\wedge\alpha)|\psi|^2\\
|\psi|^2 \langle *\alpha\psi\wedge \eta\psi\rangle= 2\Re(*\beta\wedge\omega)|\psi|^2 = 0\\
|\psi|^4\langle \eta\psi\wedge \eta\psi\rangle=-\Re(\omega\wedge\omega)|\psi|^2=0\\
|\psi|^4\langle J\eta\psi\wedge J \eta\psi\rangle=-\Re(\omega\wedge\omega)|\psi|^2=0\\
\langle *\eta\psi\wedge \eta\psi\rangle=\frac{1}{|\psi|^4}\Re(*\bar{\omega}\wedge\omega)|\psi|^2=|\psi|^2
  \end{gather*}
In the third and fourth relation we used the fact that the 2-form $\omega\wedge \omega$ takes values in the orthogonal complement in $\R^3$ of the image of the conformal $1$-form $\omega$. In the last relation, we also used the identification $\Lambda^2 TM^*\cong |K|^2$ of $2$-forms with conformal metrics on $M$. Applying those formulas, we deduce
\begin{align*}
\langle *\dbar\psi\wedge\dbar\psi\rangle&=\Re(*\bar{\alpha}\wedge \alpha)|\psi|^2+ U^2\langle *\eta\psi\wedge\eta\psi\rangle+V^2\langle J*\eta\psi\wedge J\eta\psi\rangle\\
&=|\psi|^2(|\alpha|^2+|U|^2+|V|^2)\,,
\end{align*}
which is the first identity of the lemma. The second identity follows from 
\[
\langle*\dbar\psi\wedge \eta\psi\rangle=\langle U*\eta\psi\wedge \eta\psi\rangle = U|\psi|^2
\]
and likewise does the third.
\end{proof}

As we have discussed, a nowhere vanishing section $\psi\in\Gamma(L^{\times})$ gives rise to the closed, $\R^3$ valued $1$-form $(\psi,\psi)$---and thus to a conformal immersion with translation periods---if and only if $
\int_M *\bar{\alpha}\wedge \alpha+ \int_M V^2=0
$. Here the $(0,1)$-form $\alpha\in\Gamma(\bar{K})$ and the real half-density $V\in \Gamma(|K|)$ are the components in the decomposition \eqref{eq:decomposition} of $\dbar\psi$ which, due to the previous lemma, we can express in terms of the section $\psi$.
\begin{Theorem}\label{thm:functional}
Let $L\to M$ be a spin bundle over a compact Riemann surface and denote by $\dbar$ the Dirac structure. For non-negative $\epsilon=(\epsilon_1,\epsilon_2,\epsilon_3)$ the family of functionals 
\[
E_{\epsilon}\colon \Gamma(L^{\times})\to \R
\]
on nowhere vanishing sections of $L$, given by
\[
E_{\epsilon}(\psi)=\epsilon_1\int_M \tfrac{\langle*\dbar\psi\wedge\dbar\psi\rangle}{|\psi|^2}+
(\epsilon_2-\epsilon_1)\int_M\tfrac{{\langle *\dbar\psi\wedge\psi(\psi,\psi)\rangle}^2}{|\psi|^4}+
(\epsilon_3-\epsilon_1)\int_M\tfrac{{\langle *\dbar\psi\wedge J\psi(\psi,\psi)\rangle}^2}{|\psi|^4}\,
\]
is well-defined on the Riemann surface $M$ and invariant under constant, non-zero scalings of $\psi$. In particular, one could constrain the functional to the $L^4$-sphere of sections satisfying $\int_M |\psi|^4=1$. For $\epsilon_3=0$ and arbitrary $\epsilon_1,\epsilon _2>0$ the functional assumes its minimum value zero at a section $\psi\in\Gamma(L^{\times})$, which gives rise to a conformal immersion (with translation periods) $f\colon \tilde{M}\to \R^3$ satisfying $df=(\psi,\psi)$ in the prescribed regular homotopy class given by the spin bundle $L$. 
\end{Theorem}
The proof of the theorem follows immediately from Lemma~\ref{lem:contributions}, in which the various terms of the functional $E_{\epsilon}$ are calculated. It should be noted that, in order to guarantee exactness of the closed $1$-form $(\psi, \psi)$, the functional $E_{\epsilon}$ needs to be augmented by the sum of the squares of the periods $\sum_{\gamma}|\int_{\gamma}(\psi,\psi)|^2$, where $\gamma$ ranges over a basis of the homology group $H_1(M,\Z)$.  This being said, in the sequel we will always assume that our resulting immersions are defined on $M$.

\begin{Remark}
For $\epsilon_3>0$ the functional $E_{\epsilon}$ contains  as a contribution the Willmore energy $\int_M U^2$ of the resulting immersion.
It is therefore tempting to minimize $E_{\epsilon}$ for $\epsilon_3>0$ while taking $\epsilon_3\to 0$. The resulting conformal immersion would then be a constrained Willmore surface, that is, a minimizer for the Willmore energy in a fixed conformal and regular homotopy class. At the moment there is no evidence that this strategy, which involves  $\Gamma$-convergency of our functionals, might be successful. The development of an algorithm based on Theorem~\ref{thm:functional} to carry out experiments is a work in progress. 
\end{Remark}
We finish this section with a discussion of how to adapt our variational approach to find isometric immersions $f\colon M\to\R^3$ of an oriented Riemannian surface $(M,g)$ in a given regular homotopy class described by a spin bundle $L\to M$.  Every  oriented Riemannian surface $(M,g)$ has a unique  Riemann surface structure in which $g$ is a conformal metric.  The induced metric of the immersion $f$, constructed from a nowhere vanishing section $\psi\in\Gamma(L^{\times})$ satisfying the non-linear Dirac equation, is given by 
\[
|df|^2=|(\psi,\psi)|^2=|\psi|^4\,.
\]
Hence, we need to minimize our functional under the constraint $|\psi|^4=g$ in order to find an isometric immersion. 

\begin{figure}
	\center
	\includegraphics[width=.5\textwidth]{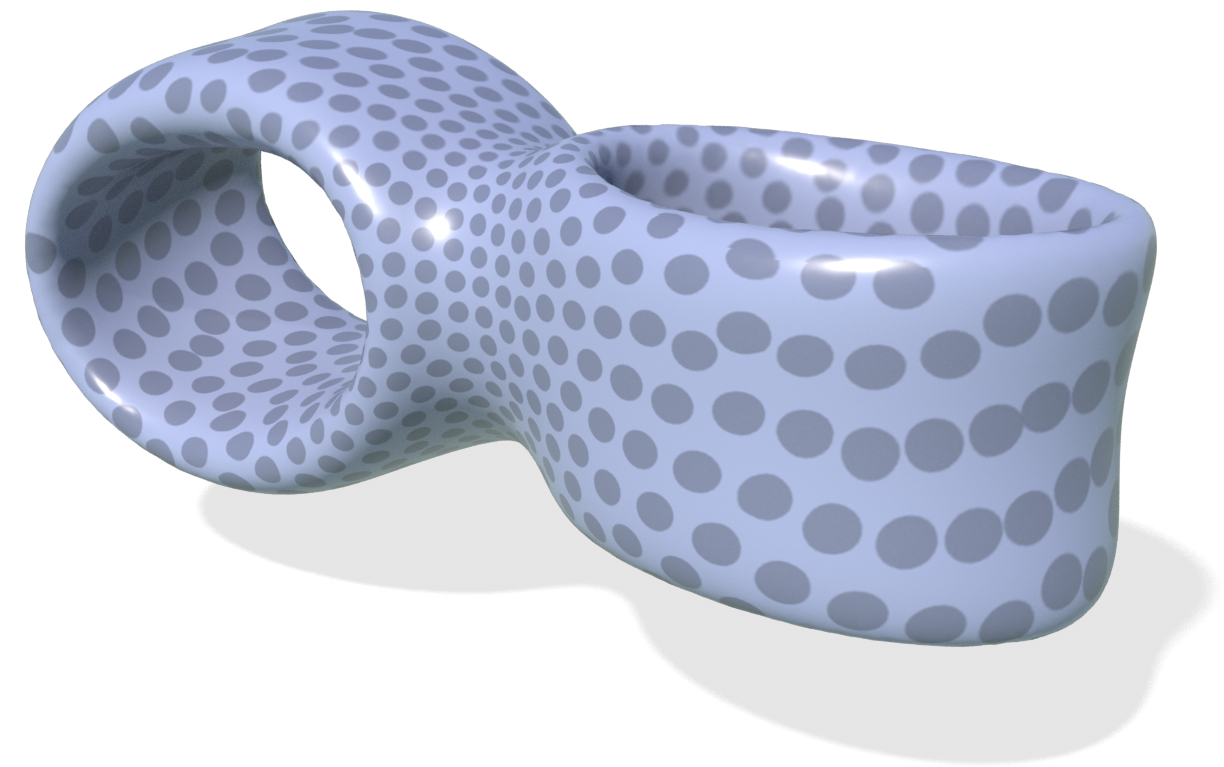}
	\caption{\label{fig:genus2-abeljac}The Abel--Jacobi map of a compact Riemann surface $M$ induces a Riemannian metric on $M$. The Gaussian curvature of this metric vanishes at the Weierstrass points. The picture shows an almost isometric smooth realization of the Abel--Jacobi metric on the abstract genus 2 surface in Figure~\ref{fig:genus2-hyperbolic} computed by the algorithm in \cite{Chern:2018:shape}. The six Weierstrass points lie on the intersection of the surface with its axis of symmetry.}
\end{figure}

\begin{Theorem}\label{thm:isometric}
Let $L\to M$ be a spin bundle over a compact, oriented Riemannian surface and denote by $\dbar$ the Dirac structure on $L$ (where we think of $M$ as a Riemann surface).  For non-negative $\epsilon=(\epsilon_1,\epsilon_2,\epsilon_3)$ the family of functionals 
\[
E_{\epsilon}\colon \Gamma(L^{\times})\to \R
\]
on nowhere vanishing sections of $L$, given by
\[
E_{\epsilon}(\psi)=\epsilon_1\int_M \tfrac{\langle*\dbar\psi\wedge\dbar\psi\rangle}{|\psi|^2}+
(\epsilon_2-\epsilon_1)\int_M\tfrac{{\langle *\dbar\psi\wedge\psi(\psi,\psi)\rangle}^2}{|\psi|^4}+
(\epsilon_3-\epsilon_1)\int_M\tfrac{{\langle *\dbar\psi\wedge J\psi(\psi,\psi)\rangle}^2}{|\psi|^4}\,
\]
subject to the constraint $|\psi|^4=g$, 
is well-defined on the Riemannian  surface $(M,g)$. For $\epsilon_3=0$ and arbitrary $\epsilon_1,\epsilon _2>0$ the functional assumes its minimum value, zero, at a section $\psi\in\Gamma(L^{\times})$ which gives rise to an isometric immersion  $f\colon M\to \R^3$ satisfying $df=(\psi,\psi)$ in the prescribed regular homotopy class given by the spin bundle $L$. 
\end{Theorem}
For a generic Riemannian surface $(M,g)$ there will not exist a smooth isometric immersion into $\R^3$, even though there always is a $C^1$ isometric immersion \cite{Nash:1954:CII, Kuiper:1955:CII}. The methods to construct $C^1$ isometric immersions result in surfaces in $\R^3$ that do not reflect the intrinsic geometry of $(M,g)$ well (see Figure~\ref{fig:C1-flattorus}). On the other hand, minimizing $E_{\epsilon}$ for non-zero $\epsilon_3>0$, that is, with the Willmore energy $\int_M |U|^2$ turned on as a  contribution to the functional, we expect the limiting isometric immersion as $\epsilon_3\to 0$ to have small Willmore energy and thus avoid excessive creasing. This has indeed been carried out experimentally with an algorithm based on Theorem~\ref{thm:isometric}, which is detailed in~~\cite{Chern:2018:shape}. These experiments give some credence to our conjecture, that there should be a piecewise smooth isometric immersion of any Riemannian surface $(M,g)$ in a given regular homotopy class. Again, a theoretical analysis of this conjecture would involve an understanding of the $\Gamma$-convergency properties of our family of functionals $E_{\epsilon}$.
\bibliographystyle{acm}

\bibliography{isometric_title_case}

\end{document}